\RequirePackage{fix-cm}
\documentclass[smallextended]{article}       
\usepackage{color}
\usepackage{amsfonts}
\usepackage{amsmath}
\usepackage{graphicx}
\usepackage{wrapfig}
\usepackage{amscd}
\usepackage{amssymb}
\usepackage{amstext}
\usepackage{mathrsfs}
\usepackage{amsthm}

 \newtheorem{theorem}{Theorem}[section]
 
 \newtheorem{lemma}{Lemma}[section]

 \newtheorem{definition}{Definition}[section]
 \newtheorem{remark}{Remark}[section] 
 \numberwithin{equation}{section}
\begin{document}

\title{Power series of the operators $U_n^{\varrho}$
}

\date{Heiner Gonska, Ioan Ra\c sa and Elena Dorina St\u anil\u a  }


\maketitle

\begin{abstract}
We study power series of members of a class of positive linear operators reproducing linear function constituting a link between genuine Bernstein-Durrmeyer and classical Bernstein operators. Using the eigenstructure of the operators we give a non-quantitative convergence result towards the inverse Voronovskaya operators. We include a quantitative statement via a smoothing approach.
\end{abstract}

\section{Introduction}
The present note is essentially motivated by two key papers of  
P\u alt\u anea which both appeared in two hardly known local Romanian  
journals.
In the first article mentioned \cite{Paltanea:2006} P\u alt\u anea 
defined power series of Bernstein operators (with $n$ fixed) and  
studied their approximation behaviour for functions defined on the  
space $C_0[0,1]:=\{f|f(x)=x(1-x)h(x), h\in C[0,1]\}$ to some extent. 
This article motivated a number  
of authors to study similar problems or give different proofs of  
P\u alt\u anea's main result. See \cite{Abel:2009},  \cite{Abel-Ivan-Paltanea:2013},  \cite{Abel-Ivan-Paltanea:2013-2},  \cite{Rasa:2012-2}. In one more and most significant article P\u alt\u anea \cite{Paltanea:2007}
introduced a very interesting link between the classical Bernstein  
operators $B_n$ and the so-called "genuine Bernstein-Durrmeyer  
operators" $U_n$, thus also bridging the gap between $U_n$ and  
piecewise linear interpolation in a most elegant way for the cases $0  
< \varrho \le 1$.
The operators $U_n^\rho$ also attracted several authors to study them  
further. See, for example, \cite{Gonska-Paltanea:2010-1}, \cite{Gonska-Paltanea:2010-2}.
In the present note we combine both approaches of P\u alt\u anea and study  
power (geometric) series of the operators $U_n^{\varrho}$, thus bridging  
the gap between power series of Bernstein operators and such of the  
genuine operators $U_n$ mentioned above.

Our main results will concern the convergence of the series as $n$ (the  
degree of the polynomials inside the series) tends to infinity. The  
first non-quantitative theorem will essentially use the  
eigenstructure of the $U_n^{\varrho}$ which was recently studied in \cite{Gonska-Rasa-Stanila:2013}.

The second result describes the degree of convergence to the 
"inverse \\ Voronovskaya operators" $-A_{\varrho}^{-1}$ using a smoothing (K- 
functional) approach and makes use of exact representations of the  
moments as presented in \cite{Gonska-Paltanea:2010-1}.

The quantitative statement also holds in the limiting case of  
Bernstein operators, thus supplementing the original work of P\u alt\u anea.

\section{The operators $ U_n^{\varrho} $ and their eigenstructure}
Denote by $C[0,1]$ the space of continuous, real-valued functions on $[0,1]$
and by $\Pi_{n}$ the space of polynomials of degree at most $n\in\mathbb{N}_{0}:=\{0,1,2...\}$.
\begin{definition}
Let $\varrho>0$ and $n\in\mathbb{N}_{0}, n\geq1$. Define the operator $U_n^{\varrho}:C[0,1]\rightarrow \Pi_{n}$
by
\begin{equation*}\label{eq1.1}
\begin{array}{lcl}
U_n^{\varrho}(f,x)&:=&\sum\limits_{k=0}^n F_{n,k}^{\varrho}(f)p_{n,k}(x)\\
&:=&\sum\limits_{k=1}^{n-1}\left(\int\limits_{0}^1\dfrac{t^{k\varrho-1}(1-t)^{(n-k)\varrho-1}}{B(k\varrho,(n-k)\varrho)}f(t)dt\right)p_{n,k}(x)+\vspace{0.3cm}\\
 &&+f(0)(1-x)^n+f(1)x^n,
\end{array}
\end{equation*}
$f\in C[0,1], x\in[0,1]$ and $B(\cdot,\cdot)$ is Euler's Beta function.
The fundamental functions $p_{n,k}$ are defined by
$$p_{n,k}(x)= {n\choose k}x^k(1-x)^{n-k}, \;\;\; 0\leq k\leq n, \;\;\; x\in[0,1].$$
\end{definition}
For $\varrho=1$ and $f\in C[0,1]$, we obtain
\begin{equation*}\label{un}
\begin{array}{lll}
U_{n}^1(f,x)=U_{n}(f,x)&=&(n-1)\sum\limits_{k=1}^{n-1}\left(\int\limits_{0}^1f(t)p_{n-2,k-1}(t)dt\right)p_{n,k}(x)\vspace{0.3cm}\\
&&+(1-x)^n f(0)+x^n f(1),
\end{array}
\end{equation*}
where $U_{n}$ are the ``genuine'' Bernstein-Durrmeyer operators (see \cite[Th. 2.3]{Gonska-Paltanea:2010-1}),
while for $\varrho\rightarrow\infty$, for each $f\in C[0,1]$ the sequence $U_n^{\varrho}(f,x)$ converges uniformly to the Bernstein polynomial
\begin{equation*}\label{ber}
B_{n}(f,x)=\sum\limits_{k=0}^{n}f\left(\dfrac{k}{n}\right)p_{n,k}(x).
\end{equation*}
Moreover, for $n$ fixed and $\varrho\rightarrow 0$ one has uniform convergence of $U_n^{\varrho}f$ towards the first Bernstein polynomial $B_1 f$, i.e., linear interpolation at $0$ and $1$ (see \cite[Th. 3.2]{Gonska-Paltanea:2010-2}).
The eigenstructure of $U_n^{\varrho}$ is described in
\cite{Gonska-Rasa-Stanila:2013}. The numbers
\begin{equation}\label{ps2.1}
\lambda_{\varrho,j}^{(n)}:=\dfrac{\varrho^jn!}{(n\varrho)^{\overline{j}}(n-j)!},
\;\;\; j=0,1,...,n,
\end{equation}
are eigenvalues of $U_n^{\varrho}$. To each of them there
corresponds a monic eigenpolynomial $p_{\varrho,j}^{(n)}$ such
that $deg\; p_{\varrho,j}^{(n)}=j, \; j=0,1,...,n.$ In particular,
\begin{equation}\label{ps2.2}
p_{\varrho,0}^{(n)}(x)=1, p_{\varrho,1}^{(n)}(x)=x-\dfrac{1}{2},
x\in[0,1].
\end{equation}
A complete description of $p_{\varrho,j}^{(n)}(x), j=2,...,n,$ can
be found in \cite{Gonska-Rasa-Stanila:2013}. From
(\cite{Gonska-Rasa-Stanila:2013}, (3.14)) we get
\begin{equation}\label{ps2.3}
p_{\varrho,j}^{(n)}(0)=p_{\varrho,j}^{(n)}(1)=0, j=2,...,n.
\end{equation}
Obviously $U_n^{\varrho}f$ can be decomposed with respect to the
basis
$\{p_{\varrho,0}^{(n)},p_{\varrho,1}^{(n)},...,p_{\varrho,n}^{(n)}\}$
of $\Pi_n$; this allows us to introduce the dual functionals
$\mu_{\varrho,j}^{(n)}:C[0,1]\rightarrow \mathbb{R}, j=0,1,...,n$,
by means of the formula
\begin{equation}\label{ps2.4}
U_n^{\varrho}f=\sum\limits_{j=0}^n\lambda_{\varrho,j}^{(n)}\mu_{\varrho,j}^{(n)}(f)p_{\varrho,j}^{(n)},
f\in C[0,1].
\end{equation}
In particular, since $U_n^{\varrho}$ restricted to $\Pi_n$ is
bijective, we have
\begin{equation}\label{ps2.5}
p=\sum\limits_{j=0}^n\mu_{\varrho,j}^{(n)}(p)p_{\varrho,j}^{(n)},
p \in \Pi_n.
\end{equation}
Now consider the numbers
\begin{equation}\label{ps2.6}
\lambda_{\varrho,j}:=-\dfrac{\varrho+1}{2\varrho}(j-1)j, j=0,1,...
\end{equation}
and the monic polynomials
\begin{equation}\label{ps2.7}
p_0^*(x)=1, p_1^*(x)=x-\dfrac{1}{2},
p_j^*(x)=x(x-1)P_{j-2}^{(1,1)}(2x-1), j\geq 2,
\end{equation}
where $P_i^{(1,1)}(x)$ are Jacobi polynomials, orthogonal with
respect to the weight $(1-x)(1+x)$ on $[-1,1], i\geq 0$. Moreover,
consider the linear functionals $\mu_j^*:C[0,1]\rightarrow
\mathbb{R}$, defined as
\begin{equation}\label{ps2.8}
\mu_0^*(f)=\dfrac{f(0)+f(1)}{2}, \mu_1^*(f)=f(1)-f(0),
\end{equation}
\begin{equation}\label{ps2.9}
\mu_j^*(f)=\dfrac{1}{2}{2j\choose
j}[(-1)^jf(0)+f(1)-j\int\limits_0^1f(x)P_{j-2}^{(1,1)}(2x-1)dx],
j\geq 2.
\end{equation}
It is easy to verify that
\begin{equation}\label{ps2.10}
\lim\limits_{n\rightarrow\infty}n(\lambda_{\varrho,j}^{(n)}-1)=\lambda_{\varrho,j},
j\geq 0.
\end{equation}
The following result can be found in
\cite{Gonska-Rasa-Stanila:2013}.
\begin{theorem}(\cite{Gonska-Rasa-Stanila:2013}) For each $j\geq
0$ we have
\begin{equation}\label{ps2.11}
\lim\limits_{n\rightarrow\infty}p_{\varrho,j}^{(n)}=p_{j}^*,\;
\mbox{uniformly on}\; [0,1],
\end{equation}
\begin{equation}\label{ps2.12}
\lim\limits_{n\rightarrow\infty}\mu_{\varrho,j}^{(n)}(p)=\mu_{j}^*(p),
 p\in \Pi.
\end{equation}
\end{theorem}
\section{The power series $ A_n^{\varrho} $ }
Consider the space
\begin{equation}\label{ps3.1}
C_0[0,1]:=\{f|f(x)=x(1-x)h(x), h\in C[0,1]\}.
\end{equation}
For $f\in C_0[0,1], f(x)=x(1-x)h(x)$, define the norm
\begin{equation}\label{ps3.2}
||f||_0:=||h||_{\infty}.
\end{equation}
Endowed with the norm $||\cdot||_0, C_0[0,1]$ is a Banach space.
Obviously,
\begin{equation}\label{ps3.3}
||f||_{\infty}\leq \dfrac{1}{4}||f||_{0}, f\in C_0[0,1].
\end{equation}
\begin{lemma}\label{lm_ps3.1}
As a linear operator on $(C_0[0,1], ||\cdot||_0), U_n^{\varrho}$
has the norm
\begin{equation}\label{ps3.4}
||U_n^{\varrho}||_{0}=\dfrac{(n-1)\varrho}{n\varrho+1}<1.
\end{equation}
\end{lemma}
\begin{proof}
Let $f\in C_0[0,1], f(x)=x(1-x)h(x), h\in C[0,1]$. By
straightforward computation we get $U_n^{\varrho}f(x)=x(1-x)u(x)$,
where
$$
u(x)=n(n-1)\sum\limits_{k=1}^{n-1}\dfrac{\int\limits_0^1t^{k\varrho}(1-t)^{(n-k)\varrho}h(t)dt}{k(n-k)B(k\varrho,(n-k)\varrho)}p_{n-2,k-1}(x).
$$
It follows immediately that $U_n^{\varrho}f\in C_0[0,1]$ and
$$
||U_n^{\varrho}f||_{0}=||u||_{\infty}\leq
\dfrac{(n-1)\varrho}{n\varrho+1}||h||_{\infty}=\dfrac{(n-1)\varrho}{n\varrho+1}||f||_0.
$$
Thus
\begin{equation}\label{ps3.5}
||U_n^{\varrho}||_{0}\leq\dfrac{(n-1)\varrho}{n\varrho+1}.
\end{equation}
On the other hand, let $g(x)=x(1-x), x\in[0,1]$. Then $||g||_0=1$
and $U_n^{\varrho}g(x)=x(1-x)\dfrac{(n-1)\varrho}{n\varrho+1}$,
which entails
$||U_n^{\varrho}g||_0=\dfrac{(n-1)\varrho}{n\varrho+1}$ and so
\begin{equation}\label{ps3.6}
||U_n^{\varrho}||_{0}\geq\dfrac{(n-1)\varrho}{n\varrho+1}.
\end{equation}
Now (\ref{ps3.4}) is a consequence of (\ref{ps3.5}) and
(\ref{ps3.6}).
\end{proof}
According to Lemma \ref{lm_ps3.1}, it is possible to consider the
operator $A_n^{\varrho}:C_0[0,1]\rightarrow C_0[0,1]$,
\begin{equation}\label{ps3.7}
A_n^{\varrho}:=\dfrac{\varrho}{n\varrho+1}\sum\limits_{k=0}^{\infty}(U_n^{\varrho})^k,
n\geq 1.
\end{equation}
For later purposes we also introduce the notation
$$
A_n^{\infty}:=\dfrac{1}{n}\sum_{k=0}^{\infty}(B_n)^k, \; n\geq 1,
$$
in order to have P\u alt\u anea's power series included.\\
 By using (\ref{ps3.4}) we get
 $||A_n^{\varrho}||_0\leq\dfrac{\varrho}{\varrho+1}$, and with the
 same function $g(x)=x(1-x)$ we find
\begin{equation}\label{ps3.8}
||A_n^{\varrho}||_0=\dfrac{\varrho}{\varrho+1}, n\geq 1.
\end{equation}
Let $p\in \Pi_m\cap C_0[0,1]$, i.e., $p(0)=p(1)=0$. Then $m\geq
2$. Let $n\geq m$. From (\ref{ps2.2}), (\ref{ps2.3}) and
(\ref{ps2.5}) we derive
\begin{equation}\label{ps3.9}
p=\sum\limits_{j=2}^m\mu_{\varrho,j}^{(n)}(p)p_{\varrho,j}^{(n)}
\end{equation}
and, moreover,
\begin{equation}\label{ps3.10}
(U_n^{\varrho})^kp=\sum\limits_{j=2}^m(\lambda_{\varrho,j}^{(n)})^k\mu_{\varrho,j}^{(n)}(p)p_{\varrho,j}^{(n)},
k\geq 0, \;\mbox{for all}\: n\geq m.
\end{equation}
According to (\ref{ps3.7}), for all $p\in \Pi_m\cap C_0[0,1]$ and
$n\geq m$,
\begin{equation}\label{ps3.11}
A_n^{\varrho}p=\dfrac{\varrho}{n\varrho+1}\sum\limits_{j=2}^m\dfrac{1}{1-\lambda_{\varrho,j}^{(n)}}\mu_{\varrho,j}^{(n)}(p)p_{\varrho,j}^{(n)}.
\end{equation}
By using (\ref{ps2.10}), (\ref{ps2.11}) and (\ref{ps2.12}) we get
\begin{equation}\label{ps3.12}
\lim\limits_{n\rightarrow\infty}A_n^{\varrho}p=\dfrac{\varrho}{\varrho+1}\sum\limits_{j=2}^m\dfrac{2}{j(j-1)}\mu_{j}^*(p)p_{j}^{*},
\end{equation}
uniformly on $[0,1]$, for all $p\in \Pi_m\cap C_0[0,1]$.
\section{The Voronovskaya operator $ A_{\varrho} $ }
It was proved in \cite[p. 918]{Gonska-Paltanea:2010-2} that
$$
\lim\limits_{n\rightarrow\infty}n(U_n^{\varrho
}g(x)-g(x))=\dfrac{\varrho+1}{2\varrho}x(1-x)g''(x), g\in C^2[0,1],
$$
uniformly on $[0,1]$. We need the following result.
\begin{theorem}
The operator $\{y\in C^2[0,1]|\;
y(0)=y(1)=0\}\rightarrow C_0[0,1]$ defined by
\begin{equation}\label{ps4.1}
A_{\varrho}y(x):=\dfrac{\varrho+1}{2\varrho}x(1-x)y''(x), x\in
[0,1],
\end{equation}
is bijective, and
\begin{equation}\label{ps4.2}
||A_{\varrho}^{-1}f||_{\infty}\leq\dfrac{\varrho}{4(\varrho+1)}||f||_0,
f\in C_0[0,1].
\end{equation}
\end{theorem}
\begin{proof}
Obviously $A_{\varrho}$ is injective. To prove the surjectivity, let $f\in C_0[0,1], f(x)=x(1-x)h(x), h\in C[0,1]$. It is a matter of calculus to verify that the function
$$
-\dfrac{2\varrho}{\varrho+1}F_{\infty}(h;x)=y(x):=-\dfrac{2\varrho}{\varrho+1}\left[(1-x)\int\limits_0^xth(t)dt+x\int\limits_x^1(1-t)h(t)dt\right], 
$$
for  $x\in [0,1]$ is in $C^2[0,1], y(0)=y(1)=0$, and $A_{\varrho}y=f$. Therefore $A_{\varrho}$ is bijective. Moreover, for $x\in [0,1]$,
$y=A_{\varrho}^{-1}(f)$, i.e., $ -y(x)=-A_{\varrho}^{-1}(f;x)=+\dfrac{2\varrho}{\varrho+1}F_{\infty}(h;x)$. Consequently,
\begin{eqnarray}
|A_{\varrho}^{-1}f(x)|&\leq&\dfrac{2\varrho}{\varrho+1}\left[(1-x)\int\limits_0^xtdt+x\int\limits_x^1(1-t)dt\right]||h||_{\infty}\\
&=&\dfrac{\varrho}{\varrho+1}x(1-x)||h||_{\infty}\leq \dfrac{\varrho}{4(\varrho+1)}||f||_{0},
\end{eqnarray}
and this leads to (\ref{ps4.2}).
\end{proof}
\begin{remark}
Further below we will use the notation $\Psi(x):=x(1-x)$, and
$$
-A_{\infty}^{-1}(\Psi h):=2\cdot F_{\infty}(h), \; h\in C[0,1],
$$
in order to also cover the Bernstein case.
\end{remark}
Another useful result reads as follows.
\begin{lemma}\label{lm_ps4.2}
For all $p\in \Pi\cap C_0[0,1]$ we have
\begin{equation}\label{ps4.3}
\lim\limits_{n\rightarrow\infty}A_n^{\varrho}p=-A_{\varrho}^{-1}p, 
\end{equation}
uniformly on $[0,1]$.
\end{lemma}
\begin{proof}
The polynomials $p_j^*$ from (\ref{ps2.7}) satisfy
\begin{equation}\label{ps4.4}
x(1-x)(p_j^*)''(x)=-j(j-1)p_j^*(x), x\in [0,1], j\geq 0
\end{equation}
(see, e.g., \cite{Cooper-Waldron:2000}, p.155).
 This yields $A_{\varrho}p_j^*=-\dfrac{\varrho+1}{2\varrho}j(j-1)p_j^*, j\geq 0,$
and, moreover,
\begin{equation}\label{ps4.5}
A_{\varrho}\left(\sum\limits_{j=2}^m\dfrac{2}{j(j-1)}\mu_{j}^{*}(p)p_{j}^{*}\right)=-\dfrac{\varrho+1}{\varrho}\sum\limits_{j=2}^m\mu_{j}^{*}(p)p_{j}^{*}
\end{equation}
 for all $p\in \Pi_m\cap C_0[0,1]$.  According to (\cite{Cooper-Waldron:2000}, (4.18)), $\sum\limits_{j=2}^m\mu_{j}^{*}(p)p_{j}^{*}=p$, so that (\ref{ps4.5}) yields
 \begin{equation}\label{ps4.6}
\dfrac{\varrho}{\varrho+1}\sum\limits_{j=2}^m\dfrac{2}{j(j-1)}\mu_{j}^{*}(p)p_{j}^{*}=-A_{\varrho}^{-1}p,
\end{equation}
for all $p\in \Pi_m\cap C_0[0,1]$.  Now (\ref{ps4.3}) is a consequence of (\ref{ps3.12}) and (\ref{ps4.6}).
\end{proof}
\section{The convergence of $ A_n^{\varrho} $ on $C_0[0,1]$}
One main result of the paper is contained in
\begin{theorem}
For all $f\in C_0[0,1]$,
 \begin{equation}\label{ps5.1}
\lim\limits_{n\rightarrow\infty}A_n^{\varrho}f=-A_{\varrho}^{-1}f,
\end{equation}
uniformly on $[0,1]$.
\end{theorem}
\begin{proof}
 Let $f\in C_0[0,1], f(x)=x(1-x)h(x), h\in C[0,1]$. Consider the polynomials $p_i(x):=x(1-x)B_ih(x)$, where $B_i$ are the classical Bernstein operators, $i\geq 1$. Then $p_i\in C_0[0,1], i\geq 1$, and $\lim\limits_{i\rightarrow\infty}||p_i-f||_0=\lim\limits_{i\rightarrow\infty}||B_ih-h||_{\infty}=0$. Let $\varepsilon>0$ and fix $i\geq 1$ such that
\begin{equation}\label{ps5.2}
||p_i-f||_0\leq \dfrac{2\varrho+2}{3\varrho+2}\varepsilon.
\end{equation}
Then, according to Lemma \ref{lm_ps4.2}, there exists $n_{\varepsilon}$ such that
\begin{equation}\label{ps5.3}
||A_n^{\varrho}p_i+A_{\varrho}^{-1}p_i||_{\infty}\leq \dfrac{2\varrho+2}{3\varrho+2}\varepsilon, n\geq n_{\varepsilon}.
\end{equation}
Now using (\ref{ps3.3}) and (\ref{ps3.8}) we infer
$$
||A_n^{\varrho}f-A_n^{\varrho}p_i||_{\infty}\leq \dfrac{1}{4}||A_n^{\varrho}f-A_n^{\varrho}p_i||_{0}\leq\dfrac{1}{4}||A_n^{\varrho}||_0||f-p_i||_{0}\leq \dfrac{\varrho}{4(\varrho+1)}\dfrac{2\varrho+2}{3\varrho+2}\varepsilon,
$$
 so that
\begin{equation}\label{ps5.4}
||A_n^{\varrho}f-A_n^{\varrho}p_i||_{\infty}\leq \dfrac{\varrho}{2(3\varrho+2)}\varepsilon.
\end{equation}
On the other hand, (\ref{ps4.2}) and (\ref{ps5.2}) yield
\begin{equation}\label{ps5.5}
||A_{\varrho}^{-1}f-A_{\varrho}^{-1}p_i||_{\infty}\leq \dfrac{\varrho}{4(\varrho+1)}||f-p_i||_0\leq\dfrac{\varrho}{2(3\varrho+2)}\varepsilon.
\end{equation}
Finally, using (\ref{ps5.3}), (\ref{ps5.4}) and (\ref{ps5.5}) we obtain, for all $n\geq n_{\varepsilon}$,
$$
||A_n^{\varrho}f+A_{\varrho}^{-1}f||_{\infty}\leq ||A_n^{\varrho}f-A_n^{\varrho}p_i||_{\infty}+||A_n^{\varrho}p_i+A_{\varrho}^{-1}p_i||_{\infty}+||A_{\varrho}^{-1}f-A_{\varrho}^{-1}p_i||_{\infty}\leq\varepsilon,
$$
and this concludes the proof.
\end{proof}
On $(C[0,1],||\cdot||_{\infty})$ consider the linear operator $H_n^{\varrho}:=A_n^{\varrho}-(-A_{\varrho}^{-1})$ given by
\begin{eqnarray*}
C[0,1]\ni h\mapsto  A_n^{\varrho}(\Psi h;x)&=&\frac{\varrho}{n\varrho+1}\sum\limits_{k=0}^{\infty}(U_n^{\varrho})^k(\Psi h;x)\in C_0[0,1]\\
C[0,1]\ni h\mapsto  -A_{\varrho}^{-1}(\Psi h;x)&=&\frac{2\varrho}{\varrho+1}\left[(1-x)\int\limits_0^x th(t)dt+x\int\limits_x^1(1-t)h(t)dt\right]\\&=&\frac{2\varrho}{\varrho+1}F_{\infty}(h;x)\in C_0[0,1]
\end{eqnarray*}
\begin{theorem}
Let $h\in C [0, 1], \varrho>0, n \geq \frac{4\varrho+6}{\varrho}, \varepsilon=\sqrt{\frac{\varrho+2}{n\varrho+2}}\leq \frac{1}{2}$ and $\Psi(x)=x(1-x)$. Then 
\begin{eqnarray}
|H_n^{\varrho}(h;x)|\leq \Psi(x)\left[\dfrac{2\varrho}{3(\varrho+1)}\sqrt{\dfrac{\varrho+2}{n\varrho+2}}\omega_1(h;\varepsilon)+ \;\;\;\;\;\;\;\;\;\;\;\;\;\;\;\;\;\;\;\;\;\;\;\;\;\;\;\;\;\;\;\;\right.\\ \left.\;\;\;\;\;\;\;\;\;\;\;\;\;\;\;\;+\dfrac{3}{4}\left(\dfrac{2\varrho}{\varrho+1}+\dfrac{2\varrho}{3(\varrho+1)}\sqrt{\dfrac{\varrho+2}{n\varrho+2}}+\dfrac{7(\varrho+3)}{6(\varrho+1)}\right)\omega_2(h;\varepsilon)\right].\nonumber
\end{eqnarray}
\end{theorem}
\begin{proof}
Let $h\in C[0,1]$ be fixed, and $g\in C^2[0,1]$ be arbitrary. \vspace{0.2cm}\\ Then
$|H_n^{\varrho}(h;x)|\leq |H_n^{\varrho}(h-g;x)|+|H_n^{\varrho}(g;x)|=|E_1|+|E_2|$.
Here
\begin{eqnarray*}
|E_1|&=&|A_n^{\varrho}(\Psi(h-g);x)-(-A_{\varrho}^{-1}(\Psi (h-g);x))|\\
&=&|A_n^{\varrho}(\Psi(h-g);x)-\frac{2\varrho}{\varrho+1}F_{\infty}(h-g;x)|\\
&\leq&||h-g||_{\infty}A_n^{\varrho}(\Psi;x)+\frac{2\varrho}{\varrho+1}|F_{\infty}(h-g;x)|\\
&=&||h-g||_{\infty}\dfrac{\varrho}{\varrho+1}\Psi(x)+\frac{2\varrho}{\varrho+1}||h-g||_{\infty}\frac{1}{2}\Psi(x)\\
&=&\frac{2\varrho}{\varrho+1}\Psi(x)|h-g||_{\infty}
\end{eqnarray*}
and
$$
|E_2|=|A_n^{\varrho}(\Psi g;x)-(-A_{\varrho}^{-1}(\Psi g;x))|.
$$
For $g\in C^2[0,1]$ one has $F_{\infty}:=F_{\infty}(g)\in C^4[0,1]$, $F''_{\infty}=-g, F'''_{\infty}=-g', F^{(4)}_{\infty}=-g''$. 
Moreover, by Taylor's formula we obtain for any points $y,t\in[0,1]$:
\begin{equation}\label{ps_g1}
F_{\infty}(t)=F_{\infty}(y)+F'_{\infty}(y)(t-y)+\frac{1}{2}F''_{\infty}(y)(t-y)^2+\frac{1}{6}F'''_{\infty}(y)(t-y)^3+\Theta_y(t)
\end{equation}
where
$$
\Theta_y(t):=\frac{1}{6}\int\limits_y^t(t-u)^3F^{(4)}_{\infty}(u)du.
$$
Fix $y$ and consider (\ref{ps_g1}) as an equality between two functions in the variable $t$. Applying to this equality the operator $U_n^{\varrho}(\cdot,y)$ one arrives at
\begin{eqnarray*}
U_n^{\varrho}(F_{\infty},y)&=&F_{\infty}(y)+\frac{1}{2}F''_{\infty}(y)U_n^{\varrho}((t-y)^2;y)+\frac{1}{6}F'''_{\infty}(y)U_n^{\varrho}((t-y)^3;y)+\\Ê&&+U_n^{\varrho}(\Theta_y;y)\\
&=& F_{\infty}(y)-\frac{1}{2}g(y)U_n^{\varrho}((t-y)^2;y)-\frac{1}{6}g'(y)(y)U_n^{\varrho}((t-y)^3;y)+\\ &&+U_n^{\varrho}(\Theta_y;y).
\end{eqnarray*}
This implies
$$
\frac{1}{2}g(y)U_n^{\varrho}((e_1-y)^2;y)-F_{\infty}(y)+U_n^{\varrho}(F_{\infty},y)=-\frac{1}{6}g'(y)U_n^{\varrho}((e_1-y)^3;y)+U_n^{\varrho}(\Theta_y;y).
$$
In the above equality we rewrite the left hand side as
$\frac{1}{2}g(y)U_n^{\varrho}((e_1-y)^2;y)-(I-U_n^{\varrho})(F_{\infty},y)$.
Thus we have
$$
g(y)U_n^{\varrho}((e_1-y)^2;y)-2(I-U_n^{\varrho})(F_{\infty},y)=-\frac{1}{3}g'(y)(y)U_n^{\varrho}((e_1-y)^3;y)+2U_n^{\varrho}(\Theta_y;y).
$$
Application of $A_n^{\varrho}$ yields
\begin{eqnarray}\label{ps_g2}
A_n^{\varrho}(g(\cdot)U_n^{\varrho}((e_1-\cdot)^2;\cdot);x)-2A_n^{\varrho}\circ(I-U_n^{\varrho})(F_{\infty},x)=\\
-\frac{1}{3}A_n^{\varrho}(g'(\cdot)U_n^{\varrho}((e_1-\cdot)^3;\cdot);x)+2A_n^{\varrho}(Q;x)\nonumber
\end{eqnarray}
where $Q(y):=U_n^{\varrho}(\Theta_y;y)$.
The first five moments are given by (see \cite[Cor. 2.1]{Gonska-Paltanea:2010-2}
\begin{eqnarray*}
U_n^{\varrho}(e_0;y)&=&1,\\
U_n^{\varrho}(e_1-y;y)&=&0,\\
U_n^{\varrho}((e_1-y)^2;y)&=&\dfrac{(\varrho+1)\Psi(y)}{n\varrho+1},\label{ch4_3.3}\\
U_n^{\varrho}((e_1-y)^3;y)&=&\dfrac{(\varrho+1)(\varrho+2)\Psi(y)\Psi'(y)}{(n\varrho+1)(n\varrho+2)},\nonumber\\
U_n^{\varrho}((e_1-y)^4;y)&=&\dfrac{3\varrho(\varrho+1)^2\Psi^2(y)n}{(n\varrho+1)(n\varrho+2)(n\varrho+3)}\nonumber\\
&&+\dfrac{-6(\varrho+1)(\varrho^2+3\varrho+3)\Psi^2(y)+(\varrho+1)(\varrho+2)(\varrho+3)\Psi(y)}{(n\varrho+1)(n\varrho+2)(n\varrho+3)}.\nonumber
\end{eqnarray*}
In the above expression we have
$2A_n^{\varrho}\circ(I-U_n^{\varrho})(F_{\infty},x)=\dfrac{2\varrho}{n\varrho+1}F_{\infty}(x)=\dfrac{2\varrho}{n\varrho+1}F_{\infty}(g;x)
.$  \\ Also
$A_n^{\varrho}(g(\cdot)U_n^{\varrho}((e_1-\cdot)^2;\cdot);x)=A_n^{\varrho}(g(\cdot)\dfrac{\varrho+1}{n\varrho+1}\Psi(\cdot);x)=\dfrac{\varrho+1}{n\varrho+1}A_n^{\varrho}(\Psi g;x)$.\\
Hence (\ref{ps_g2}) can be written as
\begin{eqnarray*}
&&\left|\dfrac{\varrho+1}{n\varrho+1}A_n^{\varrho}(\Psi g;x)-\dfrac{2\varrho}{n\varrho+1}F_{\infty}(g;x)\right|\\
&&\;\;\;\;\;\;\;\;\;\;\;\;=\left|-\frac{1}{3}g'(\cdot)U_n^{\varrho}(((e_1-\cdot)^3;\cdot);x)-2A_n^{\varrho}(Q;x)\right|\\
&&\;\;\;\;\;\;\;\;\;\;\;\;\leq \frac{1}{3}\left|A_n^{\varrho}\left(\dfrac{(\varrho+1)(\varrho+2)}{(n\varrho+1)(n\varrho+2)}\Psi'(\cdot)\Psi(\cdot);x\right)\right|+|2A_n^{\varrho}(Q;x)|\\
&&\;\;\;\;\;\;\;\;\;\;\;\;\leq \frac{1}{3}\dfrac{(\varrho+1)(\varrho+2)}{(n\varrho+1)(n\varrho+2)}||g'||_{\infty}\dfrac{\varrho}{\varrho+1}\Psi(x)+|2A_n^{\varrho}(Q;x)|.
\end{eqnarray*}
Multiplying the outermost sides of the latter inequality by $\frac{n\varrho+1}{\varrho+1}$ gives
\begin{eqnarray*}
|E_2|&=&\left|A_n^{\varrho}(\Psi g;x)-\dfrac{2\varrho}{\varrho+1}F_{\infty}(g;x)\right|\\
&\leq&\frac{\varrho(\varrho+2)}{3(n\varrho+2)(\varrho+1)}\Psi(x)||g'||_{\infty}+2\frac{n\varrho+1}{\varrho+1}|A_n^{\varrho}(Q;x)|.
\end{eqnarray*}
In the last summand we have $Q(y)=U_n^{\varrho}(\Theta_y;y)$
thus
\begin{eqnarray*}
|U_n^{\varrho}(\Theta_y;y)|&\leq&\dfrac{1}{6}U_n^{\varrho}((e_1-y)^4;y)||g''||_{\infty}\\
&\leq&\dfrac{1}{6}\cdot\dfrac{7}{4}\cdot\dfrac{(\varrho+1)(\varrho+2)(\varrho+3)}{\varrho(n\varrho+1)(n\varrho+2)}\Psi(y)||g''||_{\infty}.
\end{eqnarray*}
Hence 
\begin{eqnarray*}
\frac{2(n\varrho+1)}{\varrho+1}|A_n^{\varrho}(Q;x)|&\leq&\frac{2(n\varrho+1)}{\varrho+1}\cdot \dfrac{7}{24}\cdot\dfrac{(\varrho+1)(\varrho+2)(\varrho+3)}{\varrho(n\varrho+1)(n\varrho+2)}
A_n^{\varrho}(\Psi;x)||g''||_{\infty}\\ &=&\dfrac{7}{12}\cdot\dfrac{(\varrho+2)(\varrho+3)}{(\varrho+1)(n\varrho+2)}\Psi(x)||g''||_{\infty}.
\end{eqnarray*}
This leads to 
\begin{eqnarray*}
|E_2|&\leq&\frac{\varrho(\varrho+2)}{3(n\varrho+2)(\varrho+1)}\Psi(x)||g'||_{\infty}+\dfrac{7}{12}\cdot\dfrac{(\varrho+2)(\varrho+3)}{(\varrho+1)(n\varrho+2)}\Psi(x)||g''||_{\infty} \\
&=&\dfrac{(\varrho+2)}{3(n\varrho+2)(\varrho+1)}\Psi(x)\left\{\varrho||g'||_{\infty}+\frac{7}{4}(\varrho+3)||g''||_{\infty}\right\}.
\end{eqnarray*}
Hence for $h\in C[0,1]$ fixed, $g\in C^2[0,1]$ arbitrary we have
\begin{eqnarray*}
|H_n^{\varrho}(h;x)|=|E_1|+|E_2|\;\;\;\;\;\;\;\;\;\;\;\;\;\;\;\;\;\;\;\;\;\;\;\;\;\;\;\;\;\;\;\;\;\;\;\;\;\;\;\;\;\;\;\;\;\;\;\;\;\;\;\;\;\;\;\;\;\;\;\;\;\;\;\;\;\;\;\;\;\;\;\;\;\;\;\;\;\;\;\;\;\;\;\;\\
\leq \dfrac{2 \varrho}{\varrho+1}\Psi(x)||h-g||_{\infty}+\dfrac{(\varrho+2)}{3(n\varrho+2)(\varrho+1)}\Psi(x)\left\{\varrho||g'||_{\infty}+\frac{7}{4}(\varrho+3)||g''||_{\infty}\right\}\\
\end{eqnarray*}
Next we choose $g=h_{\varepsilon}, 0<\varepsilon=\sqrt{\frac{\varrho+2}{n \varrho+2}}\leq \frac{1}{2}$ and by applying Lemmas 2.1 and 2.4 in \cite{Gonska-Kovacheva:1994} we obtain
\begin{eqnarray*}
&&||h-g||_{\infty}\leq \frac{3}{4}\omega_2(h;\varepsilon)\\
&&||g'||\leq\frac{1}{\varepsilon}[2\omega_1(h;\varepsilon)+\frac{3}{2}\omega_2(h;\varepsilon)]\\
&&||g''||\leq\frac{3}{2\varepsilon^2}\omega_2(h;\varepsilon).
\end{eqnarray*}
Thus
\begin{eqnarray*}
|H_n^{\varrho}(h;x)|\leq \Psi(x)\left[\dfrac{2\varrho}{3(\varrho+1)}\sqrt{\dfrac{\varrho+2}{n\varrho+2}}\omega_1(h;\varepsilon)+ \;\;\;\;\;\;\;\;\;\;\;\;\;\;\;\;\;\;\;\;\;\;\;\;\;\;\;\;\;\;\;\;\right.\\ \left.\;\;\;\;\;\;\;\;\;\;\;\;\;\;\;\;+\dfrac{3}{4}\left(\dfrac{2\varrho}{\varrho+1}+\dfrac{2\varrho}{3(\varrho+1)}\sqrt{\dfrac{\varrho+2}{n\varrho+2}}+\dfrac{7(\varrho+3)}{6(\varrho+1)}\right)\omega_2(h;\varepsilon)\right].\nonumber
\end{eqnarray*}
\end{proof}
\begin{remark}
If we let $1\leq\varrho\rightarrow\infty$, then for all $n\geq 10$
\begin{eqnarray*}
\lim\limits_{\varrho\rightarrow\infty} |H_n^{\varrho}(h;x)|&=&\lim\limits_{\varrho\rightarrow\infty}|A_n^{\varrho}(\Psi h;x)-(-A_{\varrho}^{-1})(\Psi h;x)|\\
&=& |A_n^{\infty}(\Psi h;x)-(-A_{\infty}^{-1})(\Psi h;x)|\\
&\leq& 3\Psi(x)\left[\frac{1}{\sqrt{n}}\omega_1\left(h;\frac{1}{\sqrt{n}}\right)+\omega_2\left(h;\frac{1}{\sqrt{n}}\right)\right].
\end{eqnarray*}
This is a quantitative form of P\u alt\u anea's convergence result in \cite[Th. 3.2]{Paltanea:2006}.
\end{remark}



\begin{thebibliography}{}
\bibitem{Abel:2009} U. Abel: Geometric series of Bernstein-Durrmeyer operators, {\it East J. Approx.} {\bf 15} (2009), 439-450.

\bibitem{Abel-Ivan-Paltanea:2013} U. Abel, M. Ivan, R. P\u alt\u anea: Geometric series of Bernstein operators revisited, {\it J. Math. Anal. Appl.} {\bf 400} (2013), 22-24.

\bibitem{Abel-Ivan-Paltanea:2013-2} U. Abel, M. Ivan, R. P\u alt\u anea: Geometric series of positive linear operators and inverse Voronovskaya theorem, arxiv:1304.5721, 21 Apr. 2013.

\bibitem{Cooper-Waldron:2000} S. Cooper, S. Waldron: The eigenstructure of the Bernstein operator, \textit{J. Approx. Theory} {\bf 105} (2000), 133-165.

\bibitem{Gonska-Kovacheva:1994} H. Gonska, R.K. Kovacheva: The second order modulus revisited: remarks, applications, problems,  \textit{Conferenze del seminario di matematica dell'universita di Bari} {\bf 257} (1994).

\bibitem{Gonska-Paltanea:2010-1}H. Gonska, R. P\u{a}lt\u{a}nea: Simultaneous approximation by a class of Bernstein-Durrmeyer operators preserving linear functions, \emph{Czechoslovak Math. J.} {\bf 60} (2010), 783-799.

\bibitem{Gonska-Paltanea:2010-2}H. Gonska, R. P\u{a}lt\u{a}nea: Quantitative convergence theorems for a class of Bernstein-Durrmeyer operators preserving linear functions,  \emph{Ukrainian Math. J.} {\bf 62} (2010), 913-922.

\bibitem{Gonska-Rasa-Stanila:2013}H. Gonska, I. Ra\c sa, E. St\u anil\u a: The eigenstructure of operators linking the
Bernstein and the genuine Bernstein-Durrmeyer operators, {\it Mediterr. J. Math.} (2013), DOI: 10.1007/s00009-013-0347-0

\bibitem{Paltanea:2006} R. P\u{a}lt\u{a}nea: The power series of Bernstein operators, {\it Autom. Comput. Appl. Math.} {\bf 15} (2006), 247-253. 

\bibitem{Paltanea:2007} R. P\u alt\u anea: A class of Durrmeyer type operators preserving linear functions, \emph{ Ann. Tiberiu Popoviciu Sem. Funct. Equat. Approxim. Convex.} (Cluj-Napoca) {\bf 5} (2007), 109-117.

\bibitem{Rasa:2012-2} I. Ra\c sa: Power series of Bernstein operators and approximation of resolvents, {\it Mediterr. J. Math.} {\bf 9} (2012), 635-644.
\end{thebibliography}


\bigskip
 \noindent
$\begin{array}{ll}
\textrm{Heiner Gonska}\\
 \textrm{University of Duisburg-Essen} \\
 \textrm{Faculty for Mathematics} \\
 \textrm{47048 Duisburg, Germany} \\
\mathtt{heiner.gonska@uni-due.de} 
\end{array} \qquad \qquad $
 $\begin{array}{ll}
\textrm{Ioan Ra\c sa}\\
 \textrm{Technical University of Cluj-Napoca}\\
 \textrm{Department of Mathematics} \\
 \textrm{400114 Cluj-Napoca, Romania} \\
\mathtt{Ioan.Rasa@math.utcluj.ro} 
\end{array} \qquad $
\; \vspace{0.1cm}
$\begin{array}{l}
 \vspace{0.5cm}\\
\textrm{Elena-Dorina St\u anil\u a}\\
 \textrm{University of Duisburg-Essen} \\
 \textrm{Faculty for Mathematics} \\
 \textrm{47048 Duisburg, Germany} \\
\mathtt{elena.stanila@stud.uni-due.de} 
\end{array} $

\end{document}